\documentclass[12pt]{amsart}

\usepackage{amsmath}
\usepackage{amsthm}
\usepackage{amssymb}
\usepackage{amsfonts}
\usepackage{bm}
\usepackage[shortlabels]{enumitem}
\usepackage[bookmarks=true,hyperindex,pdftex,colorlinks,citecolor=red, linkcolor=blue]{hyperref}
\usepackage{centernot} %for negations
\usepackage{marginnote} % add margin notes
\usepackage{doi}
\usepackage[left=2.5cm,right=2.5cm,top=2cm,bottom=2cm]{geometry}
\usepackage{bbm}
\usepackage[all]{xy}
\usepackage{xcolor}
\usepackage{bigints}
\usepackage{cancel}
\usepackage{amsmath}
\numberwithin{equation}{section}

\theoremstyle{plain}
\newtheorem{thm}{Theorem}[section]
\newtheorem{prop}[thm]{Proposition}
\newtheorem{lemma}[thm]{Lemma}
\newtheorem{cor}[thm]{Corollary}
\newtheorem{remark}[thm]{Remark}
\newtheorem{example}[thm]{Example}
\newtheorem{nota}[thm]{Notations}

\theoremstyle{definition}
\newtheorem{df}[thm]{Definition}

\newcommand{\N}{\mathbb{N}}
\renewcommand{\textbf}[1]{\begingroup\bfseries\mathversion{bold}#1\endgroup}
\newcommand{\norm}[1]{\left\Vert #1\right\Vert}

\definecolor{darkgreen}{rgb}{.2, .6, .2}

\newcommand{\Per}{\operatorname{Per}}
\newcommand{\HC}{\operatorname{HC}}
\newcommand{\UFRec}{\operatorname{UFRec}}

\begin{document}

    \title[On Absolutely Cesàro bounded operators]{On Absolutely Cesàro bounded operators}

	\author[A. Abbar]{Arafat Abbar}
	
	\author[L. Arnold]{Loris Arnold}
	
	\author[C. Coine]{Cl\'ement Coine}
	
	\address[A. Abbar]{Department of Mathematics, École Normale Supérieure, Cadi Ayyad University, Marrakech, Morocco}
    \email{a.abbar@uca.ac.ma} 
	
	\address[C. Coine]{UNICAEN, CNRS, LMNO, 14000 Caen, France}
	\email{clement.coine@unicaen.fr}
	
	\address[L. Arnold]{UNICAEN, CNRS, LMNO, 14000 Caen, France}
	\email{lfj.arld@gmail.com}
	
	\date{}
	
	\begin{abstract}
		We study $p$-absolutely Ces\`aro bounded operators, with particular emphasis on self-improvement, weighted shifts, and linear dynamics. Our first main result shows that, for $1<p<\infty$, every positive absolutely Ces\`aro bounded operator on $L^p(\Omega)$ is automatically $p$-absolutely Ces\`aro bounded. We then characterize $q$-absolute Ces\`aro boundedness of backward shifts on weighted $\ell^p$-spaces for $q\geq p$ and use this characterization to construct examples exhibiting a wide range of possible growth rates of the powers. In the dynamical direction, we obtain new obstructions to absolute and strong Ces\`aro boundedness; in particular, no strongly Ces\`aro bounded operator on a nonzero Banach space is chaotic or upper frequently hypercyclic. We also introduce Ces\`aro ratio-boundedness and compare it with absolute Ces\`aro boundedness. Finally, a general Baire-category principle yields a genericity theorem for the failure of $p$-absolute Ces\`aro boundedness.
	\end{abstract}
	
	
	
	\maketitle
	
	\section{Introduction}
	
	Let $X$ be a Banach space, and let $\mathcal{L}(X)$ denote the Banach space of bounded linear operators on $X$. For $p>0$, an operator $T\in\mathcal{L}(X)$ is said to be \textit{$p$-absolutely Ces\`aro bounded} ($p$-ACB for short) if there exists a constant $C>0$ such that
	\[
	\sup_{n\geq 1}\frac{1}{n}\sum_{k=0}^{n-1}\|T^k x\|^p
	\leq C\|x\|^p,
	\quad \forall x\in X.
	\]
	We denote by $C_{p,\mathrm{ac}}(T)$, or simply by $C_{p,\mathrm{ac}}$ when there is no ambiguity, the smallest constant for which this estimate holds. When $p=1$, we simply say that $T$ is absolutely Ces\`aro bounded (ACB).
	
	This notion was introduced for $p=1$ by Luo and Hou in~\cite{Hou}. It is worth noting, however, that the proof of~\cite[Theorem~3.7.3]{Bel} already provides an operator satisfying this condition. The case $p=2$ had previously been investigated by van Casteren in \cite{Cast2,Cast}, while the general $p$-formulation was later considered by Cohen, Cuny, Eisner, and Lin in~\cite{Cuny1}.
	
	Berm\'udez, Bonilla, M\"uller, and Peris proved in~\cite{Mull1} that an ACB operator $T$ satisfies $\|T^n\|=o(n)$, see \cite[Corollary~2.6]{Mull1}. They also proved that, when $T$ acts on a Hilbert space, $\|T^n\|=o(\sqrt{n})$, see~\cite[Theorem~2.4]{Mull1}. These estimates were improved by Cohen, Cuny, Eisner, and Lin in~\cite{Cuny1}. More precisely, they proved that, if $T$ is $p$-ACB, then there exists $\varepsilon>0$ such that
	\(
	\|T^n\|=\mathcal{O}\bigl(n^{1/p-\varepsilon}\bigr),
	\)
	see~\cite[Proposition~3.1]{Cuny1}. In particular, for ACB operators on Hilbert spaces, they obtained
	\(
	\|T^n\|=\mathcal{O}\bigl(n^{1/2-\varepsilon}\bigr)
	\)
	for some $0<\varepsilon<1/2$; see~\cite[Theorem~4.4]{Cuny1}. We shall use the following form of \cite[Proposition~3.1]{Cuny1}; although stated there for $p\geq1$, the proof works verbatim for $p>0$.
	
	\begin{prop}\label{normpACB}
		Let $X$ be a Banach space, let $T\in\mathcal{L}(X)$, and let $0<p<\infty$. If $T$ is $p$-absolutely Ces\`aro bounded, then there exists $\varepsilon>0$ such that
		\[
		\|T^n\|=\mathcal{O}\left(n^{1/p-\varepsilon}\right).
		\]
	\end{prop}
	
	Cohen, Cuny, Eisner, and Lin also introduced and studied in~\cite{Cuny1} the weaker notion of strong Ces\`aro boundedness. We recall this notion, together with Ces\`aro boundedness.
	
	\begin{df}
		Let $T\in\mathcal{L}(X)$ and $p>0$.
		\begin{enumerate}
			\item We say that $T$ is \textit{Ces\`aro bounded} if the sequence
			\[
			M_n(T):=\frac{1}{n}\sum_{k=0}^{n-1}T^k,
			\, n\geq1,
			\]
			is bounded in $\mathcal{L}(X)$.
			
			\item We say that $T$ is \textit{$p$-strongly Ces\`aro bounded} ($p$-SCB for short) if there exists a constant $C>0$ such that
			\[
			\sup_{n\geq1}\frac{1}{n}\sum_{k=0}^{n-1}
			\left|\left\langle x^*,T^k x\right\rangle\right|^p
			\leq C\|x\|^p\|x^*\|^p,
			\qquad x\in X,\ x^*\in X^*.
			\]
			We denote by $C_{p,\mathrm{sc}}(T)$, or simply by $C_{p,\mathrm{sc}}$ when there is no ambiguity, the smallest constant for which this estimate holds. When $p=1$, we simply say that $T$ is \emph{strongly Ces\`aro bounded} (SCB).
		\end{enumerate}
	\end{df}
	
	We have the following chain of implications:
	\[
	\text{power bounded}
	\ \Longrightarrow\
	\text{ACB}
	\ \Longrightarrow\
	\text{SCB}
	\ \Longrightarrow\
	\text{Ces\`aro bounded}.
	\]
	
	The parameter $p$ in the definition above leads to a natural question: to what extent does absolute Ces\`aro boundedness depend on the exponent? A first simple observation is the following, which is a straightforward application of Jensen's inequality.
	
	\begin{prop}\label{qimpliesp}
		Let $X$ be a Banach space, let $T \in \mathcal{L}(X)$ and let $0<r \leq r'<\infty$. If $T$ is $r'$-absolutely Ces\`aro bounded (resp. $r'$-SCB), it is $r$-absolutely Ces\`aro bounded (resp. $r$-SCB).
	\end{prop}
	
	The dependence on the exponent raises a natural self-improvement problem:
	under which assumptions does ACB imply $p$-ACB for some $p>1$?
	On Hilbert spaces, every absolutely Ces\`aro bounded operator is $2$-absolutely Ces\`aro bounded; see \cite[Theorem~4.4]{Cuny1}.
	The situation on general $L^p$-spaces is different. 
	Cuny proved in \cite[Corollary~2.6]{Cuny2} that an ACB operator on $L^p(\Omega)$ satisfies a weaker averaged estimate involving the exponents $\min(p,2)$ and $\max(p,2)$, which does not yield $p$-absolute Ces\`aro boundedness when $p\neq2$.
	Moreover, for $p>2$ and $r>2$, there are ACB operators on $L^p$-spaces which fail to be $r$-ACB; see \cite[Proposition~2.8]{Cuny2}. 
	Our first main result shows that positivity restores the endpoint self-improvement: for every $1<p<\infty$, a positive ACB operator on $L^p(\Omega)$ is automatically $p$-ACB. The proof combines Krivine's inequality for positive operators on $L^p$ with the ACB estimate to derive uniform $p$-summability bounds on blocks of the orbit.
	
	The second part of the paper is devoted to a more explicit analysis of the dependence on the exponent.
	We consider backward shifts on weighted $\ell^p$-spaces, a particularly convenient class of operators for which the $p$-ACB property can be expressed directly in terms of the weight.
	This allows us to construct a wide range of examples and counterexamples. 
	In particular, we obtain a characterization of $q$-absolute Ces\`aro boundedness for $q\geq p$ in terms of a simple condition on the weight. 
	Under additional regularity assumptions on the weight, this condition is equivalent to Ces\`aro boundedness. 
	These results provide, among other things, examples showing that Ces\`aro boundedness and absolute Ces\`aro boundedness may behave differently, as well as examples illustrating different possible growth rates of the powers of the shift.
	
	We next investigate the interaction between Ces\`aro boundedness conditions and linear dynamics. One of the central concepts in linear dynamics is hypercyclicity. An operator $T\in\mathcal{L}(X)$ is said to be hypercyclic if there exists a vector $x\in X$ whose orbit under $T$, defined by
\[
\operatorname{Orb}(x,T):=\{T^n x:\ n\in\mathbb{N}_0\},
\]
is dense in $X$. Such a vector is called a hypercyclic vector for $T$, and the set of all hypercyclic vectors of $T$ is denoted by $\HC(T)$. Here, $\mathbb{N}_0=\mathbb{N}\cup\{0\}$, where $\mathbb{N}=\{1,2,\ldots\}$. The operator $T$ is said to be mixing if, for every pair of nonempty open subsets $U,V\subset X$, there exists $N\in\mathbb{N}$ such that $T^n(U)\cap V\neq\varnothing$ for every $n\geq N$. It is well known that every mixing operator is hypercyclic. If $T$ is hypercyclic and its set of periodic vectors $\Per(T):=\{x\in X:\ T^n x=x \text{ for some } n\geq 1\}$ is dense in $X$, then $T$ is called chaotic (in the sense of Devaney). We refer to \cite{BayartMatheron,GrossePeris} for general background on linear dynamics. For $x\in X$ and $U\subset X$, we define $N_T(x,U):=\{n\in\mathbb{N}_0:\ T^n x\in U\}$, called the return set from $x$ to $U$. For a set $A\subset\mathbb{N}_0$, its upper and lower densities are respectively defined by
\[
\overline{\operatorname{dens}}(A)
 :=\limsup_{N\to\infty}
 \frac{|A\cap\{0,\ldots,N-1\}|}{N}
\]
and
\[
\underline{\operatorname{dens}}(A)
 :=\liminf_{N\to\infty}
 \frac{|A\cap\{0,\ldots,N-1\}|}{N}.
\]

A vector $x\in X$ is called upper frequently recurrent for $T$ if $\overline{\operatorname{dens}}\bigl(N_T(x,U)\bigr)>0$ for every neighborhood $U$ of $x$. The set of all upper frequently recurrent vectors of $T$ is denoted by $\UFRec(T)$. If $\UFRec(T)$ is dense in $X$, then $T$ is called upper frequently recurrent \cite{BGLP}. A stronger notion is that of upper frequent hypercyclicity. A vector $x\in X$ is called upper frequently hypercyclic for $T$ if, for every nonempty open set $U\subset X$, the return set $N_T(x,U)$ has positive upper density. If such a vector exists, then $T$ is called upper frequently hypercyclic \cite{BonillaGrosse2,Shkarin}. Similarly, a vector $x\in X$ is called frequently hypercyclic for $T$ if, for every nonempty open set $U\subset X$, the return set $N_T(x,U)$ has positive lower density. An operator $T$ is called frequently hypercyclic if it has a frequently hypercyclic vector, see \cite{BayartGrivaux}. %; we refer to \cite{BayartMatheron,GrossePeris} for background. 

	Although power-bounded operators cannot be hypercyclic, absolute Ces\`aro boundedness is compatible with much richer dynamics: mixing ACB operators were constructed in \cite{Mull1}.
	This naturally leads to the problem of identifying dynamical properties that are incompatible with ACB or SCB. 
	We first prove a criterion, weaker than the Frequent Hypercyclicity Criterion (see Section~\ref{dynamics} for its statement), which rules out absolute Ces\`aro boundedness (Theorem~\ref{thm:NACB}).
	We then show that strong Ces\`aro boundedness imposes a substantially stronger obstruction: no SCB operator on a nonzero Banach space is chaotic or upper frequently hypercyclic (Corollary~\ref{cor:nochaos}). We also introduce Ces\`aro ratio-bounded operators (Definition~\ref{df:Ces-rat-bd}); every ACB operator is Ces\`aro ratio-bounded, while the converse fails even for operators satisfying the Frequent Hypercyclicity Criterion.
	
	Finally, we revisit a genericity phenomenon established in \cite[Theorem~4]{BBP}: if an operator is not ACB, then the vectors whose Ces\`aro averages of orbit norms are unbounded form a residual set. We prove a general Baire-category principle for homogeneous families of lower semicontinuous mappings and deduce the corresponding statement for $p$-ACB operators for every $p>0$; see Theorem~\ref{ThmA} and Proposition~\ref{setofvectorspacb}.

    \smallskip
	
	To summarize, the paper is organized as follows. In Section~2, we establish the self-improvement property for positive absolutely Ces\`aro bounded operators on $L^p$-spaces, for $p>1$. More precisely, we prove that a positive $1$-absolutely Ces\`aro bounded operator on $L^p(\Omega)$ is $p$-absolutely Ces\`aro bounded, and derive the corresponding consequences for the growth of its powers.
	
	Section~3 is devoted to $p$-absolutely Ces\`aro bounded backward shifts on weighted $\ell^p$-spaces. We first obtain a characterization of $q$-absolute Ces\`aro boundedness in terms of the weight. We then study additional assumptions on the weight under which this property is equivalent to Ces\`aro boundedness. Several examples and counterexamples are presented, illustrating the different possible behaviors of the powers of the shift. We conclude the section by introducing $p$-Ces\`aro ratio-bounded operators and investigating their relation with $p$-absolute Ces\`aro boundedness.
	
	In Section~4, we study the interplay between Ces\`aro boundedness conditions and linear dynamics. Besides criteria excluding absolute Ces\`aro boundedness, we prove that strongly Ces\`aro bounded operators are neither chaotic nor upper frequently hypercyclic.
	
	Finally, Section~5 contains a general Banach--Steinhaus-type result concerning the set of vectors satisfying the $p$-absolute Ces\`aro boundedness condition. As a consequence, whenever an operator is not $p$-absolutely Ces\`aro bounded, the set of vectors for which the corresponding Ces\`aro averages of orbit norms are unbounded is a dense $G_\delta$-set.

	\section{A self-improvement property on \texorpdfstring{$L^p(\Omega)$}{Lp(Omega)}}
	
	It is known that, on a Hilbert space, if $T$ is ACB, then $T$ is $2$-ACB;
	see \cite[Theorem~4.4]{Cuny1}. On $L^p(\Omega)$-spaces, it was proved in
	\cite[Corollary~2.6]{Cuny2} that if $T$ is ACB on $L^p$, $1\leq p<\infty$,
	then there exists $C_p>0$ such that, for every $N\geq1$ and every
	$x\in L^p$,
	\[
	\sum_{k=0}^{N-1}\|T^k x\|^{p_+}
	\leq
	C_p N^{p_+/p_-}\|x\|_p^{p_+},
	\]
	where
	\[
	p_-:=\min\{p,2\},
	\qquad
	p_+:=\max\{p,2\}.
	\]
	Although \cite[Corollary~2.6]{Cuny2} is stated for $\sigma$-finite measure
	spaces, this assumption is unnecessary here, since its proof only uses the
	type and cotype of $L^p$, which are the same for arbitrary measure spaces.
	This estimate is weaker than $p$-absolute Ces\`aro boundedness when
	$p\neq2$. Moreover, \cite[Proposition~2.8]{Cuny2} gives, for every $r>2$,
	the existence of an operator $T$ on $L^p(\mathbb{T})$, with $p>2$, which
	is ACB but not $r$-ACB. In particular, for $p>2$, there exist ACB
	operators which are not $p$-ACB.
	
	In this section, we show that positivity changes the situation. Namely, if $T$ is positive and ACB on $L^p$, then it is $p$-ACB. This may be viewed as the $L^p$ counterpart to \cite[Theorem 4.4]{Cuny1} in the positive setting.
	
	\begin{thm}\label{1impliesppos}
		Let $(\Omega, \mathcal{F},\mu)$ be a measure space, let $1 < p < \infty$ and let $T \colon L^p(\Omega) \to L^p(\Omega)$ be a positive operator. If $T$ is ACB, then $T$ is $p$-ACB. Moreover, there exists a constant $C_p>0$, depending only on $p$, such that
		$C_{p,\mathrm{ac}}(T) \leq C_p C_{1,\mathrm{ac}}(T)^{2p}$.
	\end{thm}
	
	The proof relies on the following standard consequence of Krivine's functional calculus for Banach lattices; see, for instance, \cite[Proposition~4.14]{Dales}.
	
	\begin{lemma}\label{lemKrivine}
		Let $S \colon L^p(\Omega) \to L^p(\Omega)$ be a positive operator. For any family of non-negative functions $f_1, \dots, f_m$ in $L^p(\Omega)$,
		\begin{equation}\label{eqKrivine}
			\left( \sum_{k=1}^m (S f_k)^p \right)^{1/p} \le S \left( \left( \sum_{k=1}^m f_k^p \right)^{1/p} \right).
		\end{equation}
	\end{lemma}
	
	\begin{proof}[Proof of Theorem \ref{1impliesppos}]
		Without loss of generality, we may assume that $\|f\|_p=1$ and that $f \ge 0$, since the positivity of $T$ yields $|T^n f| \le T^n |f|$. For every $n \ge 0$, set  
        $$a_n := \norm{T^n f}_p.$$
		Fix an integer $M \ge 2$ and a parameter $\varepsilon > 0$. For $1 \le k \le M$, define
		\begin{equation*}
			v_k := \frac{T^k f}{a_k + \varepsilon}, \quad \text{and} \quad y := \left( \sum_{k=1}^M v_k^p \right)^{1/p}.
		\end{equation*}
		Integrating $y^p$ yields
		\begin{equation}\label{norm_y}
			\norm{y}_p^p = \sum_{k=1}^M \norm{v_k}_p^p = \sum_{k=1}^M \left( \frac{a_k}{a_k + \varepsilon} \right)^p \le M.
		\end{equation}
		Set
		\[
		u := \sum_{j=0}^{M-1} T^j y = \sum_{j=0}^{M-1} T^j \Biggl( \sum_{k=1}^M \Biggl( \frac{T^k f}{a_k+\varepsilon} \Biggr)^p \Biggr)^{1/p}.
		\]
		The ACB hypothesis applied to $y$ and inequality \eqref{norm_y} give
		\begin{equation}\label{eqnormp_u}
			\norm{u}_p^p \le \left(\sum_{j=0}^{M-1} \norm{T^j y}_p\right)^p \le C_{1,\mathrm{ac}}^p M^p \norm{y}_p^p \leq C_{1,\mathrm{ac}}^p M^{p+1}.
		\end{equation}
		On the other hand, by \eqref{eqKrivine}, the function $u$ satisfies
		\[
		u \ge \sum_{j=0}^{M-1} \left( \sum_{k=1}^M \left( \frac{T^{j+k} f}{a_k + \varepsilon} \right)^p \right)^{1/p} = \sum_{j=0}^{M-1} \left( \sum_{n=1}^{2M-1} b_{j,n}^p \right)^{1/p}
		\]
		where  $b_{j,n}$ is defined for every $j \in \{0, \dots, M-1\}$ and every $n\geq 1$ by
		\[
		b_{j,n} :=
		\begin{cases}
			\dfrac{T^n f}{a_{n-j} + \varepsilon}, & \text{if } 1 \le n-j \le M, \\[3mm]
			0, & \text{otherwise.}
		\end{cases}
		\]
		By Minkowski's inequality in $\ell^p$,
		\[
		u \ge \left( \sum_{n=1}^{2M-1} \left( \sum_{j=0}^{M-1} b_{j,n} \right)^p \right)^{1/p} = \left( \sum_{n=1}^{2M-1} (T^n f)^p H_n^p \right)^{1/p},
		\]
		where
		\begin{equation*}
			H_n := \sum_{k=\max(1, \, n-M+1)}^{\min(M, \, n)} \frac{1}{a_k + \varepsilon}.
		\end{equation*}
		Taking the $L^p$-norm, we obtain
		\begin{equation}\label{eqnorm_u_lower}
			\norm{u}_p^p \ge \sum_{n=1}^{2M-1} a_n^p H_n^p.
		\end{equation}
		Note that for every $n$ in the range
		\[
		M \le n \le M + \lfloor M/2 \rfloor - 1,
		\]
		the summation index set defining $H_n$ contains the upper half $I_M := \{\lceil M/2 \rceil, \dots, M\}$. Thus, for such $n$,
		\[
		H_n\ge \sum_{k \in I_M} \frac{1}{a_k + \varepsilon}.
		\]
		By the Cauchy--Schwarz inequality,
		\[
		|I_M|^2 = \left( \sum_{k \in I_M} 1 \right)^2 \le \left( \sum_{k \in I_M} (a_k + \varepsilon) \right) \left( \sum_{k \in I_M} \frac{1}{a_k + \varepsilon} \right).
		\]
		The ACB hypothesis yields
		\[
		\sum_{k \in I_M} a_k \le \sum_{k=0}^M a_k \le C_{1,\mathrm{ac}} (M+1)
		\]
		and since $|I_M| \ge M/2$,
		\begin{equation}\label{eqHn_bound}
			H_n \ge \frac{(M/2)^2}{C_{1,\mathrm{ac}}(M+1) + M \varepsilon}.
		\end{equation}
		Restricting the sum in \eqref{eqnorm_u_lower} to $[M,M+ \lfloor M/2 \rfloor - 1]$ and using \eqref{eqHn_bound}, we obtain
		\[
		\norm{u}_p^p \ge \frac{(M/2)^{2p}}{(C_{1,\mathrm{ac}}(M+1) + M \varepsilon)^p} \sum_{n=M}^{M+\lfloor M/2 \rfloor - 1} a_n^p.
		\]
		Comparing this inequality with the upper bound \eqref{eqnormp_u}, we obtain
		$$
		\sum_{n=M}^{M+\lfloor M/2 \rfloor - 1} \norm{T^n f}_p^p = \sum_{n=M}^{M+\lfloor M/2 \rfloor - 1} a_n^p \le \frac{(C_{1,\mathrm{ac}}(M+1) + M \varepsilon)^p}{(M/2)^{2p}} C_{1,\mathrm{ac}}^p M^{p+1}.
		$$
		Letting $\varepsilon \rightarrow 0$ gives
		\begin{equation}\label{mainEsti}
			\sum_{n=M}^{M+\lfloor M/2 \rfloor - 1} \norm{T^n f}_p^p \le 4^p C_{1,\mathrm{ac}}^{2p} \frac{(M+1)^pM^{p+1}}{M^{2p}} \leq 8^p C_{1,\mathrm{ac}}^{2p} M.
		\end{equation}
		Let $(M_r)_{r \ge 0}$ be the increasing sequence of integers defined by $M_0 = 2$ and $M_{r+1} = M_r + \left\lfloor \frac{M_r}{2} \right\rfloor$. It is easy to check that $M_r \sim C \left(\frac{3}{2}\right)^r$ for some constant $C$ from which we deduce the existence of $D>0$ such that, for every integer $L$,
		\begin{equation}\label{mainEstiseq}
			\sum_{r=0}^L M_r \le DM_{L}.
		\end{equation}
		For an integer $N \geq 2$, we let $R \ge 0$ be the unique integer such that $M_R \le N < M_{R+1}$. Applying the block estimate \eqref{mainEsti} to each interval $[M_r, M_{r+1})$ together with \eqref{mainEstiseq}, we obtain
		\begin{align*}
			\sum_{n=2}^{N-1} \norm{T^n f}_p^p &\le \sum_{r=0}^R \left( \sum_{n=M_r}^{M_{r+1}-1} \norm{T^n f}_p^p \right) \\
			& \le 8^p C_{1,\mathrm{ac}}^{2p} \sum_{r=0}^R M_r\\
			&\le 8^p C_{1,\mathrm{ac}}^{2p} DM_R\\
			&\leq 8^p C_{1,\mathrm{ac}}^{2p} D N.
		\end{align*}
		Since $C_{1,\mathrm{ac}}\geq1$ and the terms corresponding to $n=0,1$ are bounded by $(1+2^p)C_{1,\mathrm{ac}}^{2p}$, division by $N$ yields the $p$-ACB estimate with a constant of the form $C_p C_{1,\mathrm{ac}}^{2p}$, where $C_p$ depends only on $p$. This concludes the proof.
	\end{proof}
	
	\begin{remark}
		Theorem \ref{1impliesppos} together with Proposition \ref{normpACB} show that if $T$ is positive on some $L^p(\Omega)$ and ACB, then $\|T^n\|=\mathcal{O}\left(n^{1/p - \varepsilon}\right)$ for some $\varepsilon>0$. This improves \cite[Corollary 2.6]{Cuny2} in the case of positive operators.
	\end{remark}
	
	\section{\texorpdfstring{$p$}{p}-absolutely Ces\`aro bounded backward shifts}
	
	Let $\omega=(\omega_k)_{k\in\N}$ be a sequence of positive real numbers (called a \textit{weight}) such that $\underset{k\in\N}{\sup}\,\dfrac{\omega_k}{\omega_{k+1}}<\infty$, let $1\leq p<\infty$, and let $\mathbb K$ denote either $\mathbb R$ or $\mathbb C$. Denote by $\ell^p(\N,\omega)$ the weighted $\ell^p$-space defined by
	$$\ell^p(\N,\omega):=\Big\{(x_k)_{k\in\N}\in \mathbb{K}^{\N}:\, \sum_{k=1}^{\infty} |x_k|^p\omega_{k}^{p}<\infty\Big\}.$$
	Equipped with the norm
	$$\|x\|_{p,\omega}:=\Big(\sum_{k=1}^{\infty} |x_k|^p\omega_{k}^{p}\Big)^{1/p},$$
	it is a Banach space. Let $B$ be the backward shift operator on $\ell^p(\N,\omega)$, that is,
	$$B\left( (x_k)_{k\in\N} \right)=(x_{k+1})_{k\in\N}.$$
	It is easy to see that
	$$
	\|B^n\|_{p,\omega} := \|B^n\|_{\mathcal{L}(\ell^p(\N,\omega))} = \underset{k\in\N}{\sup}\,\dfrac{\omega_k}{\omega_{k+n}}.
	$$
	In this section, we study the $q$-absolute Ces\`aro boundedness of $B$ as an operator on $\ell^p(\N,\omega)$ and provide several examples and counterexamples. We also give, under certain assumptions on the weight $\omega$, other necessary and sufficient conditions for the $q$-absolute Ces\`aro boundedness of $B$. A characterization of $p$-absolutely Ces\`aro bounded weighted shifts on $\ell^p$ is available, see \cite[Corollary 58]{BBP2} (there, the shift is weighted, not the $\ell^p$-space). Our Theorem \ref{TheMainDisGen} is more general.
	
	In the literature, weighted backward shifts are usually considered on the classical spaces
	$\ell^p(\N)$; see, for instance, \cite{Salas95,GrosseErdmann00}. We recall that such operators
	can be identified, up to an isometric conjugacy, with backward shifts on weighted
	$\ell^p$-spaces; see, for instance, \cite{GrossePeris}.\\

	We first record a simple invariance under equivalent weights.
	
	\begin{df}\label{equivweights}
		Two weights $\omega = (\omega_n)_n$ and $\omega' = (\omega'_n)_n$ are equivalent if there exist $0<a \leq b < \infty$ such that
		$$\forall n\in \N, \quad a\omega_n \leq \omega'_n \leq b \omega_n.$$
	\end{df}
	
	When two weights are equivalent, the associated weighted $\ell^p$-spaces are isomorphic. This immediately yields the following.
	
	\begin{lemma}\label{equivweightsimpliesequivCesaro}
		Let $\omega = (\omega_n)_n$ and $\omega' = (\omega'_n)_n$ be two equivalent weights.
		Then, for every $1 \leq p<\infty$ and $r>0$, $B$ is $r$-absolutely Ces\`aro bounded (respectively, Ces\`aro bounded) on $\ell^p(\N,\omega)$ if and only if it is $r$-absolutely Ces\`aro bounded (respectively, Ces\`aro bounded) on $\ell^p(\N,\omega')$.
	\end{lemma}
	
	\subsection{Characterization of absolutely Ces\`aro bounded backward shifts}
	
	Let us introduce the following useful notations and quantities.
	
	\begin{nota} Let $\omega=(\omega_k)_{k\in\N}$ be a weight.
		For every $r>0$, denote
		$$
		a_r:=\underset{N\geq 1}{\sup}\,\underset{i\ge N}{\sup}\,\dfrac{1}{N}\sum_{k=0}^{N-1}\,\Big(\dfrac{\omega_{i-k}}{\omega_i}\Big)^{r}, \quad \text{and} \quad c_r:=\sup_{N\ge 1}\frac{1}{N}\sum_{k=1}^{N}\,\Big(\dfrac{\omega_{k}}{\omega_{N}}\Big)^{r}.
		$$
	\end{nota}
	
	We now prove the following characterization. It generalizes \cite[Corollary 58]{BBP2}, and its statement is slightly simpler than that of the latter.
	
	\begin{thm}\label{TheMainDisGen}
		Let $B$ be the backward shift operator on $\ell^p(\N,\omega)$.
		\begin{enumerate}
			\item Let $q\ge p$ and assume that $a_q<\infty$. Then $B$ is
			$q$-absolutely Ces\`aro bounded.
			
			\item Conversely, if $B$ is $r$-absolutely Ces\`aro bounded for some
			$r>0$, then $a_r<\infty$.
		\end{enumerate}
		
		\noindent In particular,
		\begin{enumerate}[resume]
			\item If $q\ge p$, then $B$ is $q$-absolutely Ces\`aro bounded if and
			only if $a_q<\infty$. Moreover, in that case,
			\[
			C_{q,\mathrm{ac}}=a_q.
			\]
		\end{enumerate}
	\end{thm}
	
	\begin{proof}
		First of all, let us denote, in this proof,
		$$
		b_r:=\underset{N\geq 2}{\sup}\,\underset{1\leq i\leq N-1}{\max}\, \dfrac{1}{N}\, \sum_{k=0}^{i-1} \Big(\dfrac{\omega_{i-k}}{\omega_i}\Big)^{r}
		$$
		and let us show that $b_r \leq a_r$. Fix $N\geq 2$. Then, for every $1\leq i \leq N-1$,
		$$
		\dfrac{1}{N}\, \sum_{k=0}^{i-1} \Big(\dfrac{\omega_{i-k}}{\omega_i}\Big)^{r} \leq \dfrac{1}{i}\, \sum_{k=0}^{i-1} \Big(\dfrac{\omega_{i-k}}{\omega_i}\Big)^{r} \leq \underset{j\geq 1}{\sup}\,\dfrac{1}{j}\, \sum_{k=0}^{j-1} \Big(\dfrac{\omega_{j-k}}{\omega_j}\Big)^{r} \leq a_r,
		$$
		It follows that
		\begin{equation}\label{brleqar}
			b_r \leq a_r.
		\end{equation}
		
		Now let $x \in \ell^p(\N,\omega) $ with $\|x\|_{p,\omega}= 1$. Assume that $q\ge p$ and that $a_q<\infty$. For every $N\geq 1$, we have
		\begin{align*}
			\sum_{k=0}^{N-1}\|B^k x\|_{p,\omega}^{q}
			=\sum_{k=0}^{N-1} \left(\sum_{i=1}^{\infty}|x_{i+k}|^{p} \omega_{i}^{p}\right)^{q/p}
			& = \sum_{k=0}^{N-1} \left(\sum_{i=k+1}^{\infty}|x_{i}|^{p} \omega_{i-k}^{p}\right)^{q/p} \\
			&=\sum_{k=0}^{N-1} \left(\sum_{i=k+1}^{\infty}\omega_i^p|x_{i}|^{p} \frac{\omega_{i-k}^{p}}{\omega_i^p} \right)^{q/p}.
		\end{align*}
		Since, for each $k\in \N_0$, $\sum_{i=k+1}^{\infty}|x_i|^p\omega_i^p\leq1$, Jensen's inequality (after completing these coefficients to a probability measure by adding a mass at $0$ if necessary) gives
		$$
		\sum_{k=0}^{N-1} \Big(\sum_{i=k+1}^{\infty}\omega_i^p|x_{i}|^{p} \frac{\omega_{i-k}^{p}}{\omega_i^p} \Big)^{q/p} \le \sum_{k=0}^{N-1} \sum_{i=k+1}^{\infty}\omega_i^p|x_{i}|^{p}\frac{\omega_{i-k}^{q}}{\omega_i^q}.
		$$
		Thus,
		\begin{align*}\label{IneAbsCeslp_q}
			\sum_{k=0}^{N-1}\|B^k x\|_{p,\omega}^{q}&\le \sum_{i=1}^{\infty} \omega_i^p|x_{i}|^{p}\sum_{k=0}^{\min\{i-1,N-1\}}\frac{\omega_{i-k}^{q}}{\omega_i^q}  \nonumber  \\
			&\le \sum_{i=1}^{N-1}\omega_i^p|x_{i}|^{p}\sum_{k=0}^{i-1}\frac{\omega_{i-k}^{q}}{\omega_i^q}+\sum_{i=N}^{\infty}\omega_i^p|x_{i}|^{p}\sum_{k=0}^{N-1}\frac{\omega_{i-k}^{q}}{\omega_i^q}.
		\end{align*}
		It follows from \eqref{brleqar} that
		\begin{align*}
			\frac{1}{N}\sum_{k=0}^{N-1}\|B^k x\|_{p,\omega}^{q}& \le b_q\sum_{i=1}^{N-1}\omega_i^p|x_{i}|^{p}+a_q\sum_{i=N}^{\infty}\omega_i^p|x_{i}|^{p} \le a_q,
		\end{align*}
		and this proves $(1)$ and $C_{q,\mathrm{ac}}\leq a_q$.
		
		Next, let $r>0$ and assume that $B$ is $r$-absolutely Ces\`aro bounded. Denote by $(e_{k})_{k\in\N}$ the canonical basis of $\ell^p(\N,\omega)$. Let $N\in \N$. By the definition of $r$-absolute  Ces\`aro boundedness, there exists $C>0$ such that
		\begin{align*}
			CN\|x\|_{p,\omega}^r \geq \sum_{k=0}^{N-1} \|B^kx\|^r_{p,\omega},\quad  x\in \ell^p(\N,\omega), \  N\in\N.
		\end{align*}
		In particular, taking $x=\dfrac{1}{\omega_i}e_i$ for $i \geq N$, we get
		$$CN \ge\sum_{k=0}^{N-1} \dfrac{1}{\omega^r_i}\|B^ke_i\|^r_{p,\omega}=\sum_{k=0}^{N-1} \Big(\dfrac{\omega_{i-k}}{\omega_i}\Big)^r.
		$$
		This inequality holds for every $N\geq 1$ and every $i\geq N$, which yields $a_r \leq C$. This proves $(2)$.
		
		Finally, $(3)$ follows from $(1)$ and $(2)$.
	\end{proof}
	
	The preceding theorem yields the following corollary.
	
	\begin{cor}\label{iifpACB}
		Let $q\ge p\ge1$ and let $B$ be the backward shift operator on $\ell^p(\N,\omega)$. Assume that there exists $K\geq1$ such that for every $N \in \N$ and every $k \in \{0,\ldots, N-1\}$,
		\begin{equation}\label{Assumpwk}
			\sup_{i\ge N} \frac{\omega_{i-k}}{\omega_{i}} \leq K\frac{\omega_{N-k}}{\omega_{N}}.
		\end{equation}
		Then $B$ is $q$-absolutely Ces\`aro bounded if and only if $c_q < + \infty$.
		Moreover, in that case,
		$$
		\forall n\in \mathbb{N}, \quad \frac{\omega_1}{\omega_{n+1}} \leq \|B^n\|_{p,\omega} \leq \frac{K\omega_1}{\omega_{n+1}}.
		$$
		In particular, if for every $n\in \mathbb{N}$, $\omega_n=f(n)$ where $f:[1,\infty)\to(0,\infty)$ is log-convex, then $\omega=(\omega_n)_n$ satisfies \eqref{Assumpwk} with $K=1$.
	\end{cor}
	
	\begin{proof}
		
		By the definitions and \eqref{Assumpwk},
		$$c_q \leq a_q \leq K^q c_q,$$
		so that, according to Theorem \ref{TheMainDisGen}, $B$ is $q$-absolutely Ces\`aro bounded if and only if $c_q < \infty$.
		
		Moreover, for every $n\ge 1$,
		\begin{equation*}
			\frac{\omega_1}{\omega_{n+1}}\leq\|B^n\|_{p,\omega} = \sup_{i\in \N}\frac{\omega_{i}}{\omega_{i+n}} = \sup_{i>n}{\frac{\omega_{i-n}}{\omega_{i}}} \leq K \frac{\omega_1}{\omega_{n+1}}.
		\end{equation*}
		Assume now that $\omega_n = f(n)$ where $f$ is log-convex. Then for every $y \geq 1$, the function $x \mapsto \frac{\ln(f(x))-\ln(f(y))}{x-y}$ is increasing on $[1,\infty) \setminus \{y\}$. Hence, if $N \in \N$ and $k \in \{1,\ldots, N-1\}$, then for every $i\geq N$,
		$$
		\dfrac{\ln(f(N))-\ln(f(N-k))}{N-(N-k)} \leq \dfrac{\ln(f(i))-\ln(f(i-k))}{i-(i-k)}
		$$
		which gives
		$$
		\frac{\omega_N}{\omega_{N-k}} \leq \frac{\omega_i}{\omega_{i-k}}
		$$
		and yields \eqref{Assumpwk} with $K=1$.
	\end{proof}
	
	\begin{example}
		In general, the finiteness of $c_1$ does not imply the absolute Ces\`aro boundedness of $B$. Indeed, consider, for $0 < \alpha<1$ the following weight: set $\omega_1 = 1$ and for every $j\geq 1$ and $k\geq 2$ such that $10^{j-1} + 1 \leq k \leq 10^j$, let
		\[
		\omega_k = \left\{
		\begin{array}{cl}
			\dfrac{1}{k^{\alpha}}
			& \text{if } 10^{j-1} + 1 \leq k \leq 10^j - j, \\[2mm]
			\dfrac{(10^j-k+1)^{\alpha}}{10^{\alpha j}}
			& \text{if } 10^{j} - j + 1 \leq k \leq 10^j.
		\end{array}
		\right.
		\]
		In this example, we consider $B$ acting on $\ell^1(\N,\omega)$.
		
		Let us show first that $a_1 = \infty$. Indeed, let $j\geq 1$, $N=j$ and $i=10^j$. We have
		\begin{align*}
			\frac{1}{N} \sum_{k=0}^{N-1} \frac{\omega_{i-k}}{\omega_{i}}
			= \frac{1}{j} \sum_{k=0}^{j-1} \frac{\omega_{10^j-k}}{\omega_{10^j}}
			& = \frac{1}{j} \sum_{k=0}^{j-1} (k+1)^{\alpha} \underset{j \to + \infty}{\sim} \frac{1}{j} \frac{j^{\alpha+1}}{1+\alpha} = \frac{j^{\alpha}}{1+\alpha},
		\end{align*}
		so that the supremum over $j\geq 1$ is infinite and hence $a_1 = \infty$. Therefore $B$ is not absolutely Ces\`aro bounded.
		
		Next, let us show that $c_1 < \infty$. Let $N=10^j$ for some $j\geq 1$. We have
		\begin{align*}
			\sum_{k=1}^{10^j} \omega_{k}
			& \leq 1 + \sum_{k=2}^{10^j} \dfrac{1}{k^{\alpha}} + \sum_{l=1}^j \sum_{k=10^{l}-l+1}^{10^l} \dfrac{(10^l-k+1)^{\alpha}}{10^{\alpha l}} \\
			& \leq 1 + \frac{10^{(1-\alpha)j}}{1-\alpha} + \sum_{l=1}^j \sum_{k=10^{l}-l+1}^{10^l} \dfrac{l^{\alpha}}{10^{\alpha l}} \\
			& = 1 + \frac{10^{j}10^{-\alpha j}}{1-\alpha} + \sum_{l=1}^j \dfrac{l^{\alpha+1}}{10^{\alpha l}} \\
			& \leq 1 + \frac{10^{j}10^{-\alpha j}}{1-\alpha} + Mj,
		\end{align*}
		for some positive constant $M$, since $\left(\dfrac{l^{\alpha+1}}{10^{\alpha l}}\right)_{l\geq 1}$ is bounded (eventually decreasing).	We deduce that
		$$
		\sup_{j\geq 1} \frac{1}{10^j} \sum_{k=1}^{10^j} \frac{\omega_{k}}{\omega_{10^j}} \leq  \sup_{j\geq 1} \frac{10^{\alpha j}}{10^j} \left( 1 + \frac{10^{j}10^{-\alpha j}}{1-\alpha} + M j \right) < \infty
		$$
		since $\alpha < 1$. Now let $10^{j-1}<N\leq10^j$. If $N\leq10^j-j$, then $\omega_N=N^{-\alpha}$ and the same estimates give $\sum_{k=1}^N\omega_k\leq C N^{1-\alpha}$, hence
		\[
		\frac1N\sum_{k=1}^N\frac{\omega_k}{\omega_N}\leq C.
		\]
		If $10^j-j+1\leq N\leq10^j$, write $L=10^j-N+1\in\{1,\ldots,j\}$. Then $\omega_N=L^\alpha10^{-\alpha j}$ and $N\geq10^j-j+1$, while the preceding estimate gives $\sum_{k=1}^N\omega_k\leq C10^{(1-\alpha)j}$. Therefore
		\[
		\frac1N\sum_{k=1}^N\frac{\omega_k}{\omega_N}
		\leq C\frac{10^j}{N L^\alpha}\leq C',
		\]
		where the last constant is independent of $j$ and $N$. Thus $c_1<\infty$.
		
		Finally, let us prove that for every $n \in \mathbb{N}$, we have
		\[
		\|B^n\|_{1,\omega} = (1+n)^{\alpha}.
		\]
		Fix $n\in \N$ and recall that $\|B^n\|_{1,\omega} = \sup_{k\geq 1} \frac{\omega_k}{\omega_{k+n}}$. To prove the lower bound, we choose $j \geq n+1$. Consider $k = 10^j - n$. Since $n \leq j-1$, we have $10^j - j + 1 \leq k \leq 10^j$. As $k+n = 10^{j}$, the definition of the weight yields
		\[
		\omega_k = \frac{(10^j - k + 1)^\alpha}{10^{\alpha j}} = \frac{(n+1)^\alpha}{10^{\alpha j}} \quad \text{and} \quad \omega_{k+n} = 10^{-\alpha j},
		\]
		so that
		\[
		\|B^n\|_{1,\omega} \geq \frac{\omega_{10^j-n}}{\omega_{10^j}} = \frac{(n+1)^\alpha 10^{-\alpha j}}{10^{-\alpha j}} = (1+n)^\alpha.
		\]
		To establish the upper bound, we first observe that
		\begin{equation}\label{LowerBoundOmega_m}
			\omega_m \geq \frac{1}{m^{\alpha}}
		\end{equation}
		holds for all $m \geq 1$. We now estimate the ratio $\frac{\omega_k}{\omega_{k+n}}$ for any $k \geq 1$:
		\begin{itemize}
			\item If $10^j-j+1 \leq k \leq 10^j$ for some $j\in \mathbb{N}$, we set $L = 10^j-k+1 \in \{1, \dots, j\}$, so that $\omega_k = L^\alpha 10^{-\alpha j}$.
			\begin{itemize}
				\item If $k+n \leq 10^j$, then $10^j-j+1 \leq k+n \leq 10^j $, meaning $\omega_{k+n} = (L-n)^\alpha 10^{-\alpha j}$. Thus, $\frac{\omega_k}{\omega_{k+n}} = \left(\frac{L}{L-n}\right)^\alpha$. Since $L > n$, this is maximized at $L=n+1$, yielding
				\[
				\frac{\omega_k}{\omega_{k+n}} \leq (1+n)^\alpha.
				\]
				\item  If $k+n > 10^j$, then $n > 10^j - k = L-1$, which implies $n - L + 1 \geq 1$. By \eqref{LowerBoundOmega_m}, we have $\omega_{k+n} \geq (k+n)^{-\alpha}$, and noting that $k+n = 10^j + n - L + 1$, we obtain
				\[
				\frac{\omega_k}{\omega_{k+n}} = \frac{L^{\alpha}}{\omega_{k+n}10^{\alpha j}}\leq \left(\frac{L(k+n)}{10^j} \right)^{\alpha} =\left(\frac{L(10^j + n - L + 1)}{10^j}\right)^\alpha.
				\]
				Since $L \leq 10^j$ and $n - L + 1 \geq 1$, we have $L(n - L + 1) \leq 10^j(n - L + 1)$, which yields
				\begin{align*}
					L(10^j + n - L + 1) &= L 10^j + L(n - L + 1) \\
					&\le L 10^j + 10^j(n - L + 1)\\
					& = 10^j(1+n).
				\end{align*}
				Hence, we get \[\frac{\omega_k}{\omega_{k+n}} \leq (1+n)^\alpha.\]
			\end{itemize}
			\item Otherwise, $\omega_k = k^{-\alpha}$. Using \eqref{LowerBoundOmega_m}, we immediately obtain
			\[
			\frac{\omega_k}{\omega_{k+n}} \leq \frac{k^{-\alpha}}{(k+n)^{-\alpha}} = \left(1 + \frac{n}{k}\right)^\alpha \leq (1+n)^\alpha.
			\]
		\end{itemize}
		Taking the supremum over $k \geq 1$, we conclude that $\|B^n\|_{1,\omega} \leq (1+n)^{\alpha}$.
		
	\end{example}
	
	\subsection{Equivalence between Ces\`aro boundedness properties}
	
	\begin{prop}\label{iifpACB2}
		Let $1\leq p<\infty$ and let $B$ be the backward shift operator on $\ell^p(\N,\omega)$. Assume that $\omega$ is equivalent to a nonincreasing weight. Then the following assertions are equivalent:
		\begin{enumerate}
			\item $B$ is $p$-absolutely Ces\`aro bounded.
			\item $B$ is Ces\`aro bounded.
		\end{enumerate}
	\end{prop}
	
	\begin{proof}
		$(1) \implies (2)$ follows from the fact that if $B$ is $p$-absolutely Ces\`aro bounded, then it is absolutely Ces\`aro bounded by Proposition \ref{qimpliesp}, and this implies that $B$ is Ces\`aro bounded.\\
		$(2) \implies (1)$. By Lemma \ref{equivweightsimpliesequivCesaro}, we can assume that the weight $\omega$ is nonincreasing. Assume that $B$ is Ces\`aro bounded on $\ell^p(\N,\omega)$. Let $i,N \in \N$ with $i\geq N$ and define $\displaystyle x_{N,i} = \frac{1}{N^{1/p}}\sum_{k=i}^{N+i-1}\frac{e_k}{\omega_k}$. Then $\|x_{N,i}\|_{p,\omega} = 1$ and applying the definition of Ces\`aro boundedness yields the existence of $C>0$ such that
		$$
		C^p \geq \frac{1}{(2N)^p} \norm{\sum_{j=0}^{2N-1}B^jx_{N,i}}^p_{p,\omega} = \frac{1}{2^pN^{p+1}} \norm{\sum_{j=0}^{2N-1} \left(\sum_{k=i}^{N+i-1} \frac{e_{k-j}}{\omega_k}\right)}^p_{p,\omega},
		$$
		where $e_s=0$ whenever $s \leq 0$. Using the fact that $\omega$ is nonincreasing, we get
		\begin{align*}
			C^p
			& \geq \frac{1}{2^pN^{p+1}} \frac{1}{\omega_i^p} \norm{\sum_{j=0}^{2N-1} \left(\sum_{k=i}^{N+i-1} e_{k-j} \right)}^p_{p,\omega} \\
			& = \frac{1}{2^pN^{p+1}} \frac{1}{\omega_i^p} \norm{\sum_{j=0}^{2N-1} \left(\sum_{l=j+1-N}^{j} e_{i-l} \right)}^p_{p,\omega} \quad (l = i+j-k) \\
			& = \frac{1}{2^pN^{p+1}} \frac{1}{\omega_i^p} \norm{\sum_{l=1-N}^{2N-1} e_{i-l} \left(\sum_{j=\max(0,l)}^{\min(2N-1,l+N-1)} 1 \right)}^p_{p,\omega} \\
			& \geq \frac{1}{2^pN^{p+1}} \frac{1}{\omega_i^p} \norm{\sum_{l=0}^{N-1} e_{i-l} \left(\sum_{j=l}^{l+N-1} 1 \right)}^p_{p,\omega} \\
			& = \frac{1}{2^pN} \sum_{l=0}^{N-1} \left(\frac{\omega_{i-l}}{\omega_i}\right)^p
		\end{align*}
		Taking the supremum over $N\geq 1$ and $i\geq N$ gives $a_p \leq 2^p C^p$, so that $B$ is $p$-absolutely Ces\`aro bounded by Theorem \ref{TheMainDisGen}.
	\end{proof}
	
	\begin{example}
		Let $1<p<\infty$ and fix $0<\alpha<p-1$. Let
		\[
		E:=\{2^j:j\geq 0\},
		\qquad
		d(n):=\operatorname{dist}(n,E).
		\]
		Define the weight $(\omega_n)_n$ by
		\[
		\omega_n^p := v_n =(1+d(n))^\alpha.
		\]
		We will show that $B$ is Ces\`aro bounded on $\ell^p(\N,\omega)$ but not $p$-ACB. We first show the Ces\`aro boundedness. To do so, note that the Ces\`aro means of the backward shift $B$ are given by
		\[
		(M_N(B)x)_n
		= \frac{1}{N}\sum_{k=0}^{N-1}x_{n+k}.
		\]
		Pointwise, we have
		\[
		|(M_N(B)x)_n|
		\leq (M^+x)_n,
		\]
		where $M^+$ is the maximal operator defined by
		\[
		(M^+x)_n = \sup_{m\geq 1} \frac{1}{m}\sum_{k=0}^{m-1}|x_{n+k}|.
		\]
		Hence, it suffices to show that $M^+$ is bounded, that is, there exists $C\geq 0$ such that, for every $x\in \ell^p(\N,\omega)$,
		$$
		\|M^+x\|_{p,\omega} \leq C \|x\|_{p,\omega}.
		$$
		To do so, we will prove that $(v_n)_n$ satisfies
		\begin{equation}\label{estimucken}
			\sup_I
			\left(
			\frac{1}{|I|}\sum_{n\in I}v_n
			\right)
			\left(
			\frac{1}{|I|}
			\sum_{n\in I}v_n^{-1/(p-1)}
			\right)^{p-1}
			<\infty,
		\end{equation}
		where the supremum is taken over all finite intervals
		$I\subset\mathbb{N}$. This is the discrete $A_p$ condition and, in particular, implies the corresponding one-sided $A_p$ condition. The one-sided weighted maximal theorem from \cite{MartinReyes}, applied to the piecewise constant extensions of the sequence and the weight, then yields the boundedness of $M^+$.
		
		Set
		\[
		\beta:=\frac{\alpha}{p-1}\in(0,1).
		\]
		Notice that
		\[
		v_n^{-1/(p-1)}
		=(1+d(n))^{-\beta}.
		\]
		We need the following lemma.
		
		\begin{lemma}\label{estimasumdya}
			There exists a constant $C_\beta>0$ such that, for every finite
			interval $I\subset\mathbb{N}$,
			\begin{equation*}
				\sum_{n\in I}(1+d(n))^{-\beta}
				\leq C_\beta |I|^{1-\beta}.
			\end{equation*}
		\end{lemma}
		
		\begin{proof}
			For each $j\geq 0$, let
			\[
			J_j=[2^j,2^{j+1})\cap\mathbb{N}.
			\]
			We first consider an interval $K\subset J_j$ of cardinality
			$m=|K|$. Since the function
			$t\mapsto (1+t)^{-\beta}$ is decreasing, the largest possible
			value of the sum is obtained when $K$ is located next to one of the
			endpoints of $J_j$. Hence
			\begin{equation}\label{estdyad}
				\sum_{n\in K}(1+d(n))^{-\beta}
				\leq
				2\sum_{k=0}^{m-1}(1+k)^{-\beta} \leq C_\beta m^{1-\beta} = C_\beta |K|^{1-\beta}.
			\end{equation}
			We now consider an arbitrary finite interval $I\subset\mathbb{N}$.
			Let
			\[
			I_r=I\cap J_r,
			\qquad
			m_r=|I_r|.
			\]
			Only finitely many of the $I_r$ are nonempty, say $I_{r_1}, \ldots, I_{r_k}$, and
			\[
			|I|=\sum_{i=1}^k m_{r_i}.
			\]
			By \eqref{estdyad},
			\begin{equation*}
				\sum_{n\in I}(1+d(n))^{-\beta}
				\leq
				C_\beta\sum_{i=1}^k m_{r_i}^{1-\beta}.
			\end{equation*}
			If $k\leq2$, concavity gives
			\[
			\sum_{i=1}^k m_{r_i}^{1-\beta}\leq 2^\beta |I|^{1-\beta}.
			\]
			If $k\geq3$, the penultimate dyadic block is entirely contained in $I$; since the indices $r_i$ are consecutive, $r_{k-1}=r_k-1$ and
			$I_{r_{k-1}}=[2^{r_k-1},2^{r_k})$. Hence $|I|\geq2^{r_k-1}$, and
			\begin{align*}
				\sum_{i=1}^k m_{r_i}^{1-\beta}
				\leq \sum_{i=1}^k 2^{r_i(1-\beta)}
				\leq \sum_{i=0}^{r_k} 2^{i(1-\beta)}
				\leq \frac{2^{2(1-\beta)}}{2^{1-\beta}-1} (2^{r_k-1})^{1-\beta} \leq C_{\beta}'|I|^{1-\beta},
			\end{align*}
			whence the result.
		\end{proof}
		
		\noindent We now prove \eqref{estimucken}. Let $I$ be an interval of cardinality $m=|I|$, and set
		\[
		D:=\max_{n\in I}(1+d(n)).
		\]
		We distinguish two cases.
		
		\begin{itemize}
			\item Case 1: $D\leq 2m$. In this case,
			\[
			\frac{1}{m}\sum_{n\in I}v_n
			=\frac{1}{m}\sum_{n\in I}(1+d(n))^\alpha
			\leq D^\alpha
			\leq (2m)^\alpha.
			\]
			By Lemma \ref{estimasumdya}, we also have
			\[
			\frac{1}{m}\sum_{n\in I}
			v_n^{-1/(p-1)}
			\leq C_\beta m^{-\beta}.
			\]
			Consequently,
			\[
			\left(
			\frac{1}{m}\sum_{n\in I}v_n
			\right)
			\left(
			\frac{1}{m}\sum_{n\in I}
			v_n^{-1/(p-1)}
			\right)^{p-1}
			\leq
			(2m)^\alpha
			C_\beta^{p-1}m^{-\beta(p-1)}
			=
			2^\alpha C_\beta^{p-1},
			\]
			since $\beta(p-1)=\alpha$.
			\item Case 2: $D>2m$. Choose $n_0\in I$ such that
			\[
			1+d(n_0)=D.
			\]
			Since the function $d$ is $1$-Lipschitz, we have, for every $n\in I$,
			\[
			1+d(n)
			\geq
			(1+d(n_0))-|n-n_0|
			\geq D-m
			\geq \frac{D}{2}.
			\]
			Consequently,
			\[
			\frac{1}{m}\sum_{n\in I}v_n
			\leq D^\alpha
			\]
			and
			\[
			\frac{1}{m}\sum_{n\in I}
			v_n^{-1/(p-1)}
			\leq
			\left(\frac{D}{2}\right)^{-\beta}.
			\]
			It follows that
			\[
			\left(
			\frac{1}{m}\sum_{n\in I}v_n
			\right)
			\left(
			\frac{1}{m}\sum_{n\in I}
			v_n^{-1/(p-1)}
			\right)^{p-1}
			\leq
			D^\alpha
			\left(\frac{D}{2}\right)^{-\beta(p-1)}
			=2^\alpha.
			\]
		\end{itemize}
		
		\noindent We proved that \eqref{estimucken} is satisfied. Hence $B$ is Ces\`aro bounded. We now show that $B$ is not $p$-ACB. Fix $j\geq3$ and set
		\[
		n_j:=2^j,
		\qquad
		N_j:=2^{j-2}.
		\]
		For $0\leq k\leq N_j-1$, we have
		\[
		d(2^j-k)=k.
		\]
		Since $v_{2^j}=1$, we obtain
		\[
		\left(
		\frac{\omega_{2^j-k}}{\omega_{2^j}}
		\right)^p
		= (1+k)^\alpha.
		\]
		Consequently, the quantity $a_p$ satisfies
		\begin{equation*}
			a_p
			\geq
			\frac{1}{N_j}
			\sum_{k=0}^{N_j-1}(1+k)^\alpha
			\sim
			\frac{1}{\alpha+1}N_j^\alpha.
		\end{equation*}
		Since $N_j\to\infty$ as $j\to\infty$, we conclude that
		\[
		a_p=\infty.
		\]
		It follows from Theorem \ref{TheMainDisGen} that $B$ is not $p$-absolutely Ces\`aro bounded.
	\end{example}
	
	\subsection{Examples of \texorpdfstring{$p$}{p}-absolutely Ces\`aro bounded backward shifts}
	
	The following example is used in several papers, either to produce examples or counterexamples. We give a characterization of its $r$-absolute Ces\`aro boundedness on $\ell^p$ for every value of $r\geq 1$ and $p\geq 1$.
	
	\begin{example}\label{mainexample}
		Let $\alpha>0$, $r\geq 1$ and $(\omega_i)_{i\ge1} = \big(i^{-\alpha}\big)_{i\ge1}$. Let $B$ be the backward shift on $\ell^p(\N,\omega)$. Then, for every $n\geq 1$, $\|B^n\|_{p,\omega}=(n+1)^\alpha$ and
		$$
		B \ \text{is} \ r\text{-absolutely Ces\`aro bounded} \Longleftrightarrow r\alpha<1 \ \text{ and } \ p\alpha<1.
		$$
	\end{example}
	\begin{proof}
		The weight satisfies the assumptions of Corollary~\ref{iifpACB} with $K=1$. In particular, $\|B^n\|_{p,\omega}=(n+1)^\alpha$. Assume first that $p\leq r<1/\alpha$. For every $N\geq1$,
		\[
		\frac1N\sum_{k=1}^N
		\Bigl(\frac{\omega_k}{\omega_N}\Bigr)^r
		=
		\frac{N^{\alpha r}}{N}\sum_{k=1}^N\frac1{k^{\alpha r}}
		\longrightarrow \frac1{1-\alpha r}.
		\]
		Hence $c_r<\infty$, and Corollary~\ref{iifpACB} implies that $B$ is $r$-absolutely Ces\`aro bounded. If instead $r\leq p<1/\alpha$, the same computation with $p$ in place of $r$ gives $c_p<\infty$. Thus $B$ is $p$-absolutely Ces\`aro bounded and hence, by Proposition~\ref{qimpliesp}, it is $r$-absolutely Ces\`aro bounded. This proves $(\Leftarrow)$.
		
		Conversely, assume that $B$ is $r$-absolutely Ces\`aro bounded. By Proposition~\ref{normpACB}, $\|B^n\|_{p,\omega}=o(n^{1/r})$, and therefore $\alpha<1/r$. It remains to show that $p\alpha<1$. If $p\alpha\geq1$, then
		\[
		\frac{N^{\alpha p}}{N}\sum_{k=1}^N\frac1{k^{\alpha p}}
		\longrightarrow \infty,
		\]
		so $c_p=\infty$. By Corollary~\ref{iifpACB}, $B$ is not $p$-absolutely Ces\`aro bounded; Proposition~\ref{iifpACB2} then shows that $B$ is not Ces\`aro bounded. In particular, $B$ is not $1$-absolutely Ces\`aro bounded and therefore cannot be $r$-absolutely Ces\`aro bounded. This contradiction proves $p\alpha<1$.
	\end{proof}
	
	\begin{remark}
		In \cite{Hou}, the authors showed that for the weight defined by $\omega_1=1$ and $\omega_n = \frac{(2n-3)!!}{(2n-2)!!}$, $B$ is absolutely Ces\`aro bounded on $\ell^1(\N,\omega)$. Note that the weight $\omega$ is equivalent to $\omega' = (1/k^{1/2})$. Hence, Example \ref{mainexample} and Lemma \ref{equivweightsimpliesequivCesaro} recover this result.
	\end{remark}
	
	\begin{example}\label{ExDisAbsCesp}
		In the following examples, all weights satisfy the log-convexity assumption in Corollary \ref{iifpACB}. Let $1\leq p<\infty$ be fixed. We consider the operator $B$ acting on $\ell^p(\N,\omega)$.
		\begin{enumerate}
			\item Let $\kappa>0$ and take $(\omega_i)_{i\in \N} = \big(\log(i+2)^{-\kappa}\big)_{i\in \N}$.
			For every $N\geq 1$ and $q\geq 1$,
			\begin{equation*}
				\frac{1}{N}\sum_{k=1}^{N}\Bigl(\frac{\omega_{k}}{\omega_N}\Bigr)^q
				=\frac{1}{N}\sum_{k=1}^{N}
				\frac{\log(N+2)^{\kappa q}}{\log(k+2)^{\kappa q}}
				\leq C_{\kappa,q},
			\end{equation*}
			where the last inequality follows from a standard integral estimate. We conclude that for every $q\ge p$, $B$ is $q$-absolutely Ces\`aro bounded on $\ell^p(\N,\omega)$, and hence $q$-absolutely Ces\`aro bounded for every $q\geq 1$. Moreover,
			$$
			\|B^n\|_{p,\omega} = \frac{\omega_1}{\omega_{n+1}} = \text{const} \cdot \log(n+3)^{\kappa}, \quad n \ge 1.
			$$
			\item Let $\beta \in (0,1)$ and define  $(\omega_i)_{i\in \N} = \Big( \exp(-(\log (i))^{\beta}) \Big)_{i\in \N}$. The same integral comparison gives $c_q<\infty$ for every $q\geq1$, and hence $B$ is $q$-absolutely Ces\`aro bounded for every $q\geq 1$ and we have
			$$
			\|B^n\|_{p,\omega} = \frac{\omega_1}{\omega_{n+1}} =  \text{const} \cdot \exp((\log (n+1))^{\beta}), \quad n \ge 1.
			$$
			\item Let $\kappa>0$ and $\varepsilon \in (0,1)$. We consider $(\omega_i)_{i\in \N} = \Big(\log(i+2)^{-\kappa}\,i^{-\frac{1-\varepsilon}{p}}\Big)_{i\in \N}$. A standard integral estimate gives $c_p<\infty$, so Corollary \ref{iifpACB} and Proposition \ref{qimpliesp} show that $B$ is $q$-absolutely Ces\`aro bounded for every $q\leq p$.
		\end{enumerate}
	\end{example}
	
	\begin{remark}
		By Proposition \ref{normpACB}, if $B$ is $q$-absolutely Ces\`aro bounded for every $q\geq 1$, then $\|B^n\|_{p,\omega} = o\left(n^{\varepsilon}\right)$ for every $\varepsilon>0$. In Example \ref{ExDisAbsCesp} (2), we found such an example for which $\|B^n\|_{p,\omega}$ grows faster than any power of $\log(n+1)$.
	\end{remark}
	
	\begin{example}
		\label{ex:adj}
		For every $r\geq 1$, there exists an $r$-SCB operator that is not $r$-ACB.
		
		\noindent Indeed, let $1<p<\infty$, and let $q$ be the conjugate exponent of $p$. Let $\alpha>0$ satisfy
		\[
		r\alpha<1
		\quad\text{and}\quad
		p\alpha<1.
		\]
		Let $B$ be the backward shift on $\ell^p(\mathbb N,\omega)$ given in Example~\ref{mainexample}. By Example~\ref{mainexample}, the operator $B$ is $r$-ACB. In particular, $B$ is $r$-SCB. Since $B$ is not power bounded, it follows from \cite[Proposition 2.1]{Cuny1} that its adjoint  $B^\ast$, acting on $\ell^q(\mathbb N,1/\overline{\omega})$, is not $r$-ACB. On the other hand, since $\ell^p$ is reflexive, the definition of $r$-SCB shows that this property is preserved under taking adjoints; hence $B^\ast$ is $r$-SCB.
	\end{example}
	
	\subsection{\texorpdfstring{$p$}{p}-Ces\`aro ratio-bounded operators}
	
	\begin{prop} \label{p-CRB}
		Let $X$ be a Banach space, let $T\in \mathcal{L}(X)$ be a non-nilpotent operator, and let $0< p<\infty$. If $T$ is $p$-absolutely Ces\`aro bounded, then
		$$\underset{n\geq 1}{\sup}\, \dfrac{1}{n}\sum_{k=0}^{n-1}\dfrac{\|T^{n-1}\|^p}{\|T^k\|^p} \leq C_{p,\mathrm{ac}}(T).$$
		
	\end{prop}
	
	\begin{proof}
		Let $x\in X$ with $\|x\|=1$ and let $n\in \mathbb{N}$. For every $k=0,\ldots, n-1$, we have $\|T^{n-1}x\| \leq \|T^k\|\|T^{n-1-k}x\|$ so
		$$\|T^{n-1-k}x\|\geq \dfrac{\|T^{n-1}x\|}{\|T^k\|}.$$
		Thus,
		$$
		\frac{1}{n} \sum_{k=0}^{n-1} \dfrac{\|T^{n-1}x\|^p}{\|T^k\|^p} \leq \frac{1}{n} \sum_{k=0}^{n-1} \|T^{n-1-k}x\|^p \leq C_{p,\mathrm{ac}}(T).
		$$
		The result follows by taking the supremum over such $x$.
	\end{proof}
	
	\begin{df}
		\label{df:Ces-rat-bd}
		Let $X$ be a Banach space. A non-nilpotent operator $T\in\mathcal{L}(X)$ is said to be $p$-Ces\`aro ratio-bounded if
		$$\underset{n\ge1}{\sup}\, \dfrac{1}{n}\sum_{k=0}^{n-1}\dfrac{\|T^{n-1}\|^p}{\|T^k\|^p}<\infty.$$
		When $p = 1$, we simply say that $T$ is  Ces\`aro ratio-bounded.
	\end{df}
	
	\begin{remark}
		Let $X$ be a Banach space, let $T\in \mathcal{L}(X)$ be a non-nilpotent operator, and let $0< p<\infty$. If $T$ is $p$-Ces\`aro ratio-bounded, then
		$$\|T^n\|=\mathcal{O}(n^{1/p-\varepsilon}),$$
		for some $0<\varepsilon \leq 1/p$. Indeed, there exists  $C>0$ such that $pC \geq \max\{1,p\}$ and
		$$\|T^{N}\|^p \leq \frac{(N+1)C}{\sum_{k=0}^{N}\|T^{k}\|^{-p}},\quad  \forall N\in\N.$$
		Now, by  \cite[Lemma 3.2]{Cuny1},
		$$\|T^N\|\leq C^{1/p} 2^{1/pC} (N+1)^{1/p-1/pC}, \quad \forall N\in\N.$$
		Hence $\|T^N\|=\mathcal{O}(N^{1/p-\varepsilon})$, for $\varepsilon=\frac{1}{pC}$.
	\end{remark}
	
	By Proposition~\ref{p-CRB}, every $p$-ACB operator is $p$-Ces\`aro ratio-bounded. However, the converse does not hold in general. Indeed, an operator is $p$-Ces\`aro ratio-bounded if and only if its adjoint is $p$-Ces\`aro ratio-bounded. Therefore, Example~\ref{ex:adj} yields, for every $p\geq1$, a $p$-Ces\`aro ratio-bounded operator which is not $p$-ACB.
	
	\section{Absolute Ces\`aro boundedness and dynamics} \label{dynamics}
	
	We now turn to dynamical obstructions to absolute Ces\`aro boundedness.
	
	\begin{thm}
		\label{thm:NACB}
		Let $X$ be a Banach space and let $T\in\mathcal{L}(X)$. Assume that there exists a sequence $(y_n)_{n\geq 1}$ in $X$ such that
		\begin{enumerate}
			\item $\underset{\varepsilon_k=\pm 1}{\sup}
			\left\|
			\sum_{k=n}^{2n}\varepsilon_k y_k
			\right\| \underset{n\to\infty}{\longrightarrow}  0$;
			\item There exists $n_0\geq 1$ such that $\underset{n\ge n_0}{\inf}\|T^n y_n\|>0$.
		\end{enumerate}
		Then $T$ is not absolutely Ces\`aro bounded.
	\end{thm}
	
	\begin{proof}
		Set $\Delta_n:=\underset{\varepsilon_k=\pm 1}{\sup}
		\left\|
		\sum_{k=n}^{2n}\varepsilon_k y_k
		\right\|$. Assume, by contradiction, that $T$ is absolutely Ces\`aro bounded. Then there exists a constant $C>0$ such that
		\[
		\sum_{k=0}^{n-1}\|T^k x\|
		\leq
		Cn\|x\|,
		\qquad \forall x\in X,\ \forall n\in\N.
		\]
		By $(2)$, there exists $\delta>0$ such that, for every $n\geq n_0$,
		\begin{equation}
			\label{ge_delta}
			\delta<\|T^n y_n\|.
		\end{equation}
		Fix $n\geq n_0$. Let $E_n
		:= \{-1,1\}^{n+1}$. For each $\varepsilon=(\varepsilon_n,\varepsilon_{n+1},\ldots,\varepsilon_{2n})\in E_n$, define
		\[
		y_\varepsilon
		:=
		\sum_{k=n}^{2n}\varepsilon_k y_k.
		\]
		We have $\|y_\varepsilon\|\leq\Delta_n$. Therefore, by absolute Ces\`aro boundedness,
		\[
		\sum_{k=n}^{2n}\|T^k y_\varepsilon\|
		\leq
		\sum_{k=0}^{2n}\|T^k y_\varepsilon\|
		\leq
		C(2n+1)\|y_\varepsilon\|
		\leq
		C(2n+1) \Delta_n, \quad \forall \varepsilon\in E_n.
		\]
		Hence
		\[
		\frac{1}{| E_n|}
		\sum_{\varepsilon\in E_n}
		\sum_{k=n}^{2n}\|T^k y_\varepsilon\|
		\leq
		C(2n+1)\Delta_n.
		\]
		Note that, for $n\leq j\leq 2n$, we have
		\[
		\frac{1}{| E_n|}
		\sum_{\varepsilon\in E_n}
		\varepsilon_j T^j y_\varepsilon
		=
		\sum_{k=n}^{2n}
		\left(
		\frac{1}{| E_n|}
		\sum_{\varepsilon\in E_n}
		\varepsilon_j\varepsilon_k
		\right)
		T^j y_k.
		\]
		But
		\[
		\frac{1}{| E_n|}
		\sum_{\varepsilon\in E_n}
		\varepsilon_j\varepsilon_k
		=
		\begin{cases}
			1 & k=j\\
			0 & k\neq j.
		\end{cases}
		\]
		Thus, for $n\leq j\leq 2n$, we have
		\[
		T^j y_j
		=
		\frac{1}{| E_n|}
		\sum_{\varepsilon\in E_n}
		\varepsilon_j T^j y_\varepsilon.
		\]
		Consequently,
		\[
		\sum_{j=n}^{2n}\|T^j y_j\|
		\leq
		\frac{1}{| E_n|}
		\sum_{\varepsilon\in E_n}
		\sum_{j=n}^{2n}\|T^j y_\varepsilon\|
		\leq
		C(2n+1)\Delta_n.
		\]
		Combining this with \eqref{ge_delta}, we get
		\[
		\delta< C\frac{2n+1}{n+1}\Delta_n.
		\]
		Since $\Delta_n \xrightarrow[]{n\to\infty} 0$, the right-hand side tends to $0$, which is impossible because $\delta> 0$. Therefore $T$ is not absolutely Ces\`aro bounded.
	\end{proof}
	
	Condition $(1)$ of Theorem \ref{thm:NACB} is weaker than the unconditional convergence of the series $\sum_{k\geq 1} y_k$. Therefore, we  obtain the following corollary:
	
	\begin{cor}
		\label{Criterion_NACB}
		Let $X$ be a Banach space, $T\in\mathcal{L}(X)$. Assume that there exist a nonzero vector $x\in X$ and a sequence $(y_k)_{k\geq1}$ in $X$ such that
		\begin{enumerate}
			\item  $\sum_{k=1}^{\infty}y_k$ converges unconditionally;
			\item $T^ny_n\xrightarrow[]{n\to\infty} x$.
		\end{enumerate}
		Then $T$ is not absolutely Ces\`aro bounded.
	\end{cor}
	
	Recall that an operator $T\in\mathcal{L}(X)$ is said to satisfy the Frequent Hypercyclicity Criterion if there exists a dense subset $D\subset X$ and a map $S\colon D\to D$ such that, for every $x\in D$,
	\begin{enumerate}
		\item  $\sum_{n=0}^{\infty}T^n x$ converges unconditionally;
		\item $\sum_{n=0}^{\infty}S^n x$ converges unconditionally;
		\item $TSx=x$.
	\end{enumerate}
	The criterion originates in \cite{BayartGrivaux}; we refer to
	\cite{BonillaGrosse} for the strengthened form used here.
	In \cite[Corollary~6]{BBP}, it is shown that no absolutely Ces\`aro bounded operator on a Banach space is mean Li--Yorke chaotic. Moreover, every operator satisfying the Frequent Hypercyclicity Criterion is mean Li--Yorke chaotic, see also \cite{BBP}. Consequently, no absolutely Ces\`aro bounded operator can satisfy the Frequent Hypercyclicity Criterion. This result can also be obtained as a consequence of Corollary \ref{Criterion_NACB}. 
	
	\begin{cor}
		\label{cor:FHC}
		Let $X$ be a Banach space and let $T\in\mathcal{L}(X)$. If $T$ satisfies the Frequent Hypercyclicity Criterion, then $T$ is not absolutely Ces\`aro bounded.
	\end{cor}
	Observe that, in order to apply Corollary~\ref{cor:FHC}, only the second and third conditions of the Frequent Hypercyclicity Criterion are needed. Since every operator satisfying the Frequent Hypercyclicity Criterion is both chaotic and upper frequently hypercyclic, Corollary~\ref{cor:nochaos} below generalizes Corollary~\ref{cor:FHC}.

	The following example shows that the conclusion of Corollary \ref{cor:FHC} does not extend to Ces\`aro ratio-bounded operators. More precisely, there exists a Ces\`aro ratio-bounded operator satisfying the Frequent Hypercyclicity Criterion and, consequently, such an operator is also mean Li--Yorke chaotic.
	
	\begin{example}
		\label{ex:p-CRB:FHC}
		Let $1<p<\infty$ and choose $\alpha$ such that $\frac1p<\alpha<1$. Let $\omega=(\omega_k)_{k\in\N}$ be the weight defined by
		\[
		\omega_k=\frac1{k^\alpha},
		\qquad k\geq1,
		\]
		and let $B$ be the backward shift on $\ell^p(\mathbb N,\omega)$. Then
		\[
		\|B^n\|_{p,\omega}
		=
		\sup_{k\geq1}\frac{\omega_k}{\omega_{k+n}}
		=(n+1)^\alpha.
		\]
		Hence
		\[
		\frac1n\sum_{k=0}^{n-1}
		\frac{\|B^{n-1}\|}{\|B^k\|}
		=
		\frac1n\sum_{k=0}^{n-1}
		\left(\frac{n}{k+1}\right)^{\alpha}.
		\]
		Since $\alpha<1$, by the integral test, there exists a constant $C>0$ such that
		\[
		\frac1n\sum_{k=0}^{n-1}
		\left(\frac{n}{k+1}\right)^{\alpha}
		\leq C,
		\qquad n\geq1.
		\]
		Thus $B$ is Ces\`aro ratio-bounded. Note that, by Example~\ref{mainexample}, $B$ is not absolutely Ces\`aro bounded.
		
		Let us show that $B$ satisfies the Frequent Hypercyclicity Criterion (FHC). Let $D=c_{00}(\N)$ be the space of finitely supported sequences. Then $D$ is dense in $\ell^p(\N,\omega)$. Define a linear map $S:D\to D$ by
		$$
		Se_k=e_{k+1}.
		$$
		Then $BSe_k=e_k$ for every $k\in\N$, and hence, by linearity, $BSx=x$ for all $x\in D$. Thus, the third condition of the FHC is satisfied.
		
		We now verify the first condition. Let $x\in D$. Since $x$ has finite support, there exists $N_0\geq 1$ such that
		$$
		B^n x=0,\qquad n\geq N_0.
		$$
		Therefore, the series $\sum_{n=0}^{\infty}B^n x$ is a finite sum and hence converges unconditionally.
		
		Let us verify the second condition.
		Fix $k\in\N$, and let $(\varepsilon_n)_{n\geq 0}$ be any sequence in $\{-1,1\}^{\N_0}$. We have
		\[\Big\|\sum_{n=0}^{\infty}\varepsilon_n S^n e_k\Big\|_{p,\omega}^{p}
		=\sum_{n=0}^{\infty}\dfrac{1}{(k+n)^{\alpha p}}
		\leq \sum_{n=1}^{\infty}\dfrac{1}{n^{\alpha p}}<\infty,\]
		since $\alpha p>1$.  Hence, $\sum_{n=0}^{\infty}\varepsilon_n S^n e_k$ converges in $\ell^p(\N,\omega)$ for every choice of  $(\varepsilon_n)_{n\ge0}$ in $\{-1,1\}^{\N_0}$. Thus, $\sum_{n=0}^{\infty}S^n e_k$ converges unconditionally in $\ell^p(\N,\omega)$. By linearity, $\sum_{n=0}^{\infty}S^n x$ converges unconditionally for every $x\in D$.
		
		Consequently, $B$ satisfies the Frequent Hypercyclicity Criterion. In particular, $B$ is mean Li--Yorke chaotic.
		
	\end{example}
	
	If, in Corollary \ref{Criterion_NACB}, we replace unconditional convergence by absolute convergence, then we can even deduce that the operator is not Ces\`aro bounded.
	\begin{thm}
		\label{Criterion_NCB}
		Let $X$ be a Banach space and let $T\in\mathcal{L}(X)$. Assume that there exists  a sequence $(y_k)_{k\geq1}$ in $X$ such that
		\begin{enumerate}
			\item  $\sum_{k=1}^{\infty}y_k$ is absolutely convergent;
			\item There exists $n_0\geq 1$ such that $\underset{n\ge n_0}{\inf}\|T^n y_n\|>0$.
		\end{enumerate}
		Then $\sum_{n=1}^{\infty} \frac{1}{\|T^n\|} <  \infty$ and  $\sup_{n\in\N}\dfrac{\|T^n\|}{n}=\infty$. In particular,  $T$ is not  Ces\`aro bounded.
	\end{thm}
	\begin{proof}
		By $(2)$, there exists $\delta>0$ such that $\delta \leq  \|T^ny_n\| \leq \|T^n\| \|y_n\|$ for all $n\ge n_0$. Hence, by $(1)$, we get
		$$\sum_{n=1}^{\infty} \frac{1}{\|T^n\|} <  \infty.$$
		Now, aiming for a contradiction, assume that $\|T^n\|=\mathcal{O}(n)$. There then exists $C>0$  such that
		$$\|T^n\|\leq C n,\quad n\geq 1.$$
		Thus
		$$\|y_n\|\geq \dfrac{\delta}{C}\times\dfrac{1}{n},$$
		which contradicts the condition $(1)$.

        Finally, if $T$ were Ces\`aro bounded, then it would satisfy $\|T^n\| = \mathcal{O}(n)$, see e.g. \cite{Mull1}. This yields the conclusion.
	\end{proof}
	
	Let $X$ be a Banach space and let $T\in\mathcal{L}(X)$. We denote by $Z(T)$ the set of vectors whose orbit admits $0$ as an
	accumulation point, that is,
	\[
	Z(T)
	:=
	\left\{
	x\in X:
	\liminf_{n\to\infty}\|T^n x\|=0
	\right\}.
	\]

	\begin{thm}
		\label{th:noch:ufr}
		Let $X$ be a Banach space, let $T\in\mathcal{L}(X)$ be strongly Ces\`aro bounded, and set
		\[
		W(T):=
		\left\{
		x\in X:
		\frac{1}{N}\sum_{n=0}^{N-1}
		\left|\left\langle x^*,T^n x\right\rangle\right|
		\longrightarrow 0
		\text{ for every }x^*\in X^*
		\right\}.
		\]
		Then the following statements hold:
		\begin{enumerate}
			\item $W(T)$ is a closed linear subspace of $X$;
			\item $Z(T)\subset W(T)$;
			\item $W(T)\cap \UFRec(T)=\{0\}$.
		\end{enumerate}
		Consequently, if $Z(T)$ is dense in $X$, then
		\[
		W(T)=X,\qquad \UFRec(T)=\{0\},\qquad \Per(T)=\{0\}.
		\]
	\end{thm}
	
	\begin{proof}
		$(1)$. It is clear that $W(T)$ is a linear subspace of $X$. It remains to prove that $W(T)$ is closed. Let $(x_j)_{j\ge 1}$ be a sequence in $W(T)$ converging to some $x\in X$. Fix $x^*\in X^*$. By the triangle inequality and the strong Ces\`aro boundedness of $T$, for every $j\ge1$ and every $N\geq1$, we obtain
		\begin{align*}
			\frac1N\sum_{n=0}^{N-1}\left|\left\langle x^*,T^n x\right\rangle\right|
			&\leq
			\frac1N\sum_{n=0}^{N-1}\left|\left\langle x^*,T^n(x-x_j)\right\rangle\right| +\frac1N\sum_{n=0}^{N-1}\left|\left\langle x^*,T^nx_j\right\rangle\right|\\
			&\leq C_{1,\mathrm{sc}}\|x^\ast\|\,\|x-x_j\|+\frac1N\sum_{n=0}^{N-1}\left|\left\langle x^*,T^nx_j\right\rangle\right|.
		\end{align*}
		Since $x_j\in W(T)$, we get
		\[
		\limsup_{N\to\infty}\frac1N\sum_{n=0}^{N-1}\left|\left\langle x^*,T^n x\right\rangle\right|
		\leq C_{1,\mathrm{sc}}\|x^\ast\|\,\|x-x_j\|.
		\]
		Letting $j$ tend to infinity, we obtain $\limsup_{N\to\infty}\frac1N\sum_{n=0}^{N-1}\left|\left\langle x^*,T^n x\right\rangle\right|=0$. Hence $x\in W(T)$. Therefore, $W(T)$ is closed.
		
		$(2)$. Suppose  that $x\in Z(T)$ and let $x^\ast \in X^\ast$. Then there exists an increasing sequence of integers $(k_j)_{j\geq1}$  such that
		\[
		\|T^{k_j}x\|\xrightarrow[]{j\to\infty}0.
		\]
		Fix $k\geq0$. For every $N>k$, we have
		\[
		\begin{aligned}
			\frac{1}{N}\sum_{n=0}^{N-1}\left|
			\langle x^*,T^n x\rangle
			\right|
			&=
			\frac{1}{N}\sum_{n=0}^{k-1}\left|
			\langle x^*,T^n x\rangle
			\right|
			+
			\frac{1}{N}\sum_{n=k}^{N-1}\left|
			\langle x^*,T^n x\rangle
			\right|\\
			&=
			\frac{1}{N}\sum_{n=0}^{k-1}\left|
			\langle x^*,T^n x\rangle
			\right|
			+
			\frac{1}{N}\sum_{n=0}^{N-k-1}\left|
			\langle x^*,T^n(T^k x)\rangle
			\right|.
		\end{aligned}
		\]
		Since $T$ is strongly Ces\`aro bounded, we have
		\[
		\frac{1}{N-k}
		\sum_{n=0}^{N-k-1}\left|
		\langle x^*,T^n(T^k x)\rangle
		\right|
		\leq
		C_{1,\mathrm{sc}}\|x^\ast\|\|T^k x\|.
		\]
		Hence
		\[
		\begin{aligned}
			\frac{1}{N}
			\sum_{n=0}^{N-k-1}\left|
			\langle x^*,T^n(T^k x)\rangle
			\right|
			&=
			\frac{N-k}{N}
			\frac{1}{N-k}
			\sum_{n=0}^{N-k-1}\left|
			\langle x^*,T^n(T^k x)\rangle
			\right|\\
			&\leq
			\frac{N-k}{N}
			C_{1,\mathrm{sc}}\|x^\ast\|\,\|T^k x\|\\
			&\leq
			C_{1,\mathrm{sc}}\|x^\ast\|\|T^k x\|.
		\end{aligned}
		\]
		Therefore,
		\[
		\limsup_{N\to\infty}
		\frac{1}{N}\sum_{n=0}^{N-1}\left|
		\langle x^*,T^n x\rangle
		\right|
		\leq
		C_{1,\mathrm{sc}}\|x^\ast\|\,\|T^k x\|.
		\]
		Now, taking $k=k_j$ and letting $j\to\infty$, we obtain
		\[
		\limsup_{N\to\infty}
		\frac{1}{N}\sum_{n=0}^{N-1}\left|
		\langle x^*,T^n x\rangle
		\right|
		=0.
		\]
		Hence
		$$\lim_{N\to\infty}\frac{1}{N}\sum_{n=0}^{N-1}\left|
		\langle x^*,T^n x\rangle
		\right|=0.$$
		Therefore, $x\in W(T)$.
		
		$(3)$. Suppose that $0\neq x\in W(T)\cap\UFRec(T)$.  Choose
		$x^*\in X^*$ with $\langle x^*,x\rangle=1$ and set
		\[
		U:=\{y\in X:|\langle x^*,y-x\rangle|<1/2\}.
		\]
		If $n\in N_T(x,U)$, then
		$|\langle x^*,T^n x\rangle|>1/2$.  Therefore
		\[
		\limsup_{N\to\infty}\frac1N\sum_{n=0}^{N-1}
		|\langle x^*,T^n x\rangle|
		\geq \frac12\overline{\mathrm{dens}}N_T(x,U)>0,
		\]
		contrary to $x\in W(T)$.  This proves (3).
		
		Finally, if $Z(T)$ is dense, (1) and (2) give $W(T)=X$, so (3) gives
		$\UFRec(T)=\{0\}$.  Every periodic vector is upper frequently recurrent; hence $\Per(T)=\{0\}$.
	\end{proof}
	
	\begin{cor}\label{cor:nochaos}
		No strongly Ces\`aro bounded operator on a nonzero Banach space is
		chaotic.  Moreover, no strongly Ces\`aro bounded operator is upper frequently
		hypercyclic (and hence none is frequently hypercyclic).
	\end{cor}
	
	\begin{proof}
		Assume that $T$ is strongly Ces\`aro bounded. Let us first show that $T$ cannot be chaotic. Suppose, for contradiction, that $T$ is chaotic. Since $T$ is hypercyclic and $\HC(T)\subset Z(T)$, we obtain that $Z(T)$ is dense in $X$. Theorem \ref{th:noch:ufr} gives $\Per(T)=\{0\}$, which is a contradiction.
		
		Now, assume that $T$ is upper frequently hypercyclic. By \cite[Theorem 2.5]{BGLP},   $T$ admits a hypercyclic upper frequently recurrent vector. Hence $\HC(T)\cap\UFRec(T)\neq\emptyset$, but by  Theorem \ref{th:noch:ufr}, we have $\HC(T)\cap\UFRec(T)\subset W(T)\cap \UFRec(T)=\{0\}$, which is a contradiction.
	\end{proof}

    As mentioned above, there exist mixing ACB operators \cite[Corollary~2.3]{Mull1}. We strengthen this result by exhibiting a mixing operator that is absolutely strongly Kreiss bounded (ASKB). We refer to \cite{Cuny1} for the definition of ASKB operators and to \cite[Proposition~4.10]{Cuny1} for the fact that every ASKB operator is ACB. Consider the weight $\omega=(\omega_k)_{k\in\mathbb N}$ defined by
    \[
    \omega_k=\frac{1}{\log(k+1)},
    \quad k\geq1.
    \]
    The backward shift $B$ on $\ell^1(\mathbb N,\omega)$ is both absolutely strongly Kreiss bounded and mixing. Indeed, since $\omega_k\to0$ as $k\to\infty$, the operator $B$ is mixing; see, for example, \cite{GrossePeris}. Using the same argument as in the proof of \cite[Proposition~4.9]{Cuny1}, we obtain that $B$ is ASKB.

	\section{On the set of vectors satisfying the \texorpdfstring{$p$}{p}-ACB condition}
	
	In this final section, we extend the genericity statement of \cite[Theorem~4]{BBP} from ACB to $p$-ACB operators for every $p>0$. The argument is an application of the following Baire-category principle for homogeneous families of lower semicontinuous mappings.
	
	\begin{prop}\label{setofvectorspacb}
		Let $X$ be a Banach space, $T\in\mathcal{L}(X)$ and $0<p<\infty$. The following assertions are equivalent:
		\begin{enumerate}
			\item $T$ is not $p$-absolutely Ces\`aro bounded.
			\item There exists a vector $x\in X$ such that
			$$\underset{N\in\N_0}{\sup}\dfrac{1}{N+1}\sum_{j=0}^{N} \|T^jx\|^p=\infty.$$
			\item The set of all vectors $x\in X$ such that
			$$\underset{N\in\N_0}{\sup}\dfrac{1}{N+1}\sum_{j=0}^{N} \|T^jx\|^p=\infty$$
			is residual in $X$.
		\end{enumerate}
	\end{prop}
	
	If $(X_1,\| \cdot \|_{X_1}),\ldots, (X_n, \| \cdot \|_{X_n})$ are Banach spaces, we let $\| . \| : X_1 \times \cdots \times X_n \to \mathbb{R}_+$ be the product norm, defined for every $(x_1,\ldots, x_n)$ by
	$$
	\|(x_1,\ldots,x_n)\| = \max_{1\leq i \leq n} \|x_i\|_{X_i}.
	$$
	
	\begin{df}
		Let $X_1,\ldots,X_n$ be Banach spaces, $I$ an arbitrary index set, and let $(\varphi_\alpha)_{\alpha\in I}$ be a family of mappings from $X_1\times\cdots\times X_n$ into $[0,\infty)$. We say that $(\varphi_\alpha)_{\alpha\in I}$ is admissible if it satisfies the following:
		\begin{itemize}
			\item For every $\alpha\in I$, $\varphi_\alpha$ is lower semicontinuous;
			\item There exists $r>0$ such that, for every $\lambda \geq 0$, every $x \in X_1\times\cdots\times X_n $ and every $\alpha\in I$,
			$$\varphi_\alpha(\lambda x)=\lambda^r \varphi_\alpha(x);$$
			\item There exists $C>0$ such that for every $(x_1,\ldots,x_n)\in X_1\times\cdots\times X_n $, every $(y_1,\ldots,y_n)\in X_1\times\cdots\times X_n$ and every $\alpha\in I$,
			$$\varphi_\alpha(x_1-y_1,\ldots,x_n-y_n)\leq C \max_{z_i \in \{ x_i,y_i\}} \ \varphi_\alpha(z_1,\ldots,z_n),$$
			where the maximum is taken over all $2^n$ choices $z_i\in\{x_i,y_i\}$.
		\end{itemize}
	\end{df}
	
	\begin{thm}\label{ThmA}
		Let $X_1,\ldots,X_n$ be Banach spaces, $I$ an arbitrary index set, and let $(\varphi_\alpha)_{\alpha\in I}$ be an admissible family of mappings from $X_1\times\cdots\times X_n$ into $[0,\infty)$. Then the following assertions are equivalent:
		\begin{enumerate}
			\item For every $M>0$, there exist $\alpha\in I$ and $(x_1,\ldots,x_n)\in X_1\times\cdots\times X_n$, with $\|(x_1,\ldots,x_n)\|=1$ such that
			$$\varphi_\alpha(x_1,\ldots,x_n)>M.$$
			\item There exists a vector $(x_1,\ldots,x_n)\in X_1\times\cdots\times X_n$ such that
			$$\underset{\alpha\in I }{\sup}\,\varphi_\alpha(x_1,\ldots,x_n)=\infty.$$
			\item The set of all vectors $(x_1,\ldots,x_n)\in X_1\times\cdots\times X_n$ such that
			$$\underset{\alpha\in I }{\sup}\,\varphi_\alpha(x_1,\ldots,x_n)=\infty$$
			is a dense $G_\delta$-subset of $X_1\times\cdots\times X_n$.
		\end{enumerate}
	\end{thm}
	
	We first record the Baire-category lemma used in the proof. As usual, we denote the unit sphere of a Banach space $X$ by $S_X$.
	
	\begin{lemma}\label{lem_M_s}
		Under the assumptions of Theorem \ref{ThmA}, we let, for each integer $s \in \N$,
		$$M_s:=\Big\{x\in X_1\times\cdots\times X_n:\, \underset{\alpha\in I}{\sup} \ \varphi_{\alpha}(x) >s\Big\}.$$
		If for each $s \in \N$, $M_s\cap S_X\neq\emptyset$, where $X:=X_1\times\cdots\times X_n$, then the set
		$$M=\Big\{x\in X_1\times\cdots\times X_n:\, \underset{\alpha\in I}{\sup} \ \varphi_{\alpha}(x) = \infty \Big\}$$
		is a dense $G_\delta$-subset of $X_1\times\cdots\times X_n$.
	\end{lemma}
	
	\begin{proof}
		Since every $\varphi_{\alpha}$ is lower semicontinuous, the mapping $\underset{\alpha\in I}{\sup} \ \varphi_{\alpha}$ is lower semicontinuous as well so that each $M_s$ is open. Since $M=\underset{s\geq1}{\bigcap}M_s$, by the Baire category theorem, it suffices to show that every $M_s$ is dense in $X_1\times\cdots\times X_n$. Fix $s\geq1$, $x\in X$ and $\varepsilon>0$. Let $s'$ be an integer such that $s' \geq \dfrac{Cs}{(2\varepsilon)^r}$ and let $v\in M_{s'}\cap S_X$. Let $\mathcal Z$ be the finite set of the $2^n$ vectors $z=(z_1,\ldots,z_n)$ such that $z_i=x_i\pm\varepsilon v_i$. For every $\alpha\in I$, admissibility gives
		\[
		(2\varepsilon)^r\varphi_\alpha(v)
		=\varphi_\alpha((x+\varepsilon v)-(x-\varepsilon v))
		\leq C\max_{z\in\mathcal Z}\varphi_\alpha(z).
		\]
		Taking the supremum over $\alpha$ and using that $\mathcal Z$ is finite, we obtain
		\[
		\frac{(2\varepsilon)^r}{C}\sup_{\alpha\in I}\varphi_\alpha(v)
		\leq \max_{z\in\mathcal Z}\sup_{\alpha\in I}\varphi_\alpha(z).
		\]
		Since $v\in M_{s'}$, there exists $z\in\mathcal Z$ such that $\sup_{\alpha\in I}\varphi_\alpha(z)>s$. Moreover, $\|x-z\|\leq\varepsilon$. Hence $M_s$ is dense, which concludes the proof.
	\end{proof}
	
	\begin{proof}[Proof of Theorem \ref{ThmA}]
		It is straightforward that $(3) \Rightarrow (2)$, and $(2) \Rightarrow (1)$ is clear because if $x=(x_1,\ldots,x_n)\in X_1\times\cdots\times X_n$ satisfies
		$$\underset{\alpha\in I }{\sup}\,\varphi_\alpha(x)=\infty,$$
		then so does $y = \frac{x}{\|x\|}$ thanks to the $r$-homogeneity of each $\varphi_{\alpha}$.
		
		Finally, if $(1)$ is satisfied, then, using the notations of Lemma \ref{lem_M_s},  for every $s\in \N$, $M_s \cap S_X \neq \emptyset$, so by Lemma \ref{lem_M_s}, $(3)$ follows.
	\end{proof}
	
	Proposition \ref{setofvectorspacb} now follows by applying Theorem \ref{ThmA} to the family
	$$\varphi_N : X \ni x \mapsto \dfrac{1}{N+1}\sum_{j=0}^{N} \|T^jx\|^p, \quad N\in \mathbb{N}_0.$$
	Each $\varphi_N$ is continuous and $p$-homogeneous. Moreover, the inequality
	\[
	(a+b)^p\leq 2^{\max\{p-1,0\}}(a^p+b^p), \qquad a,b\geq0,
	\]
	gives
	\[
	\begin{aligned}
		\varphi_N(x-y)
		&\leq 2^{\max\{p-1,0\}}
		\bigl(\varphi_N(x)+\varphi_N(y)\bigr)\\
		&\leq 2^{\max\{p-1,0\}+1}
		\max\{\varphi_N(x),\varphi_N(y)\},
	\end{aligned}
	\]
	so the third admissibility condition holds uniformly in $N$. Hence the family is admissible.

    Theorem~\ref{ThmA} also yields the following analogue of Proposition~\ref{setofvectorspacb} for SCB operators.
    \begin{cor}\label{Cor_SCB}
    Let $X$ be a  Banach space,  $T\in\mathcal{L}(X)$ and  $p>0$.  The following assertions are equivalent:
    \begin{enumerate}
    \item $T$ is not $p$-strongly Ces\`aro bounded.
    \item  There exists  $(x,x^\ast)\in X\times X^\ast$ such that $\underset{n\geq 1 }{\sup}\,\dfrac{1}{n}\sum_{k=1}^{n}\big|  \left\langle x^*,T^k x  \right\rangle\big|^p=+\infty.$
        \item The set of all vectors $(x,x^\ast)\in X\times X^\ast$ such that
        $$\underset{n\geq 1 }{\sup}\,\dfrac{1}{n}\sum_{k=1}^{n} \big|  \left\langle x^*,T^k x  \right\rangle\big|^p=+\infty.$$
        is a dense $G_\delta$-subset of $X\times X^\ast$. 
    \end{enumerate}
\end{cor}
Note that Corollary~\ref{Cor_SCB} still holds if, in the definition of a SCB operator, we replace the iterates of $T$ by a sequence of operators. We can also apply Theorem~\ref{ThmA} to obtain a similar result for ASKB operators.

	\medskip
	
	\section*{Declaration on the use of generative AI}
	
	During the preparation of this manuscript, the authors used ChatGPT (OpenAI) for language editing and editorial assistance. The authors independently verified all mathematical statements and arguments and take full responsibility for the contents of the paper.

\end{document}